\documentclass[12pt,leqno]{amsart}

\usepackage{color}
\usepackage{enumerate}
\usepackage{graphicx}
\usepackage{epsfig}
\usepackage{amssymb,amsmath,amsthm,amsfonts}

\usepackage[letterpaper, margin=1in]{geometry} 

\usepackage {latexsym}

\usepackage{bbm}

\allowdisplaybreaks

\newcommand{\nc}{\newcommand}
\nc{\les}{\lesssim}
\nc{\nit}{\noindent}
\nc{\nn}{\nonumber}
\nc{\D}{\partial}
\nc{\diff}[2]{\frac{d #1}{d #2}}
\nc{\diffn}[3]{\frac{d^{#3} #1}{d {#2}^{#3}}}
\nc{\pdiff}[2]{\frac{\partial #1}{\partial #2}}
\nc{\pdiffn}[3]{\frac{\partial^{#3} #1}{\partial{#2}^{#3}}}
\nc{\abs}[1] {\lvert #1 \rvert}
\nc{\cAc}{{\cal A}_c}
\nc{\cE}{{\cal E}}
\nc{\cF}{{\cal F}}
\nc{\cP}{{\cal P}}
\nc{\cV}{{\cal V}}
\nc{\cQ}{{\cal Q}}
\nc{\cGin}{{\cal G}_{\rm in}}
\nc{\cGout}{{\cal G}_{\rm out}}
\nc{\cO}{{\cal O}}
\nc{\Lav}{{\cal L}_{\rm av}}
\nc{\cL}{{\cal L}}
\nc{\cB}{{\cal B}}
\nc{\cZ}{{\cal Z}}
\nc{\cR}{{\cal R}}
\nc{\cT}{{\cal T}}
\nc{\cY}{{\cal Y}}
\nc{\cX}{{\cal X}}
\nc{\cXT}{{{\cal X}(T)}}
\nc{\cBT}{{{\cal B}(T)}}
\nc{\vD}{{\vec \mathcal{D}}}
\nc{\efield}{\mathcal{E}}
\nc{\mE}{\mathcal{E}}
\nc{\vE}{{\vec \efield}}
\nc{\vB}{{\vec \mathcal{B}}}
\nc{\vH}{{\vec \mathcal{H}}}
\nc{\F}{  \mathcal{F} }
\nc{\ty}{{\tilde y}}
\nc{\tu}{{\tilde u}}
\nc{\tV}{{\tilde V}}
\nc{\Pc}{{\bf P_c}}
\nc{\bx}{{\bf x}}
\nc{\bX}{{\bf X}}
\nc{\bXYZ}{{\bf XYZ}}
\nc{\bY}{{\bf Y}}
\nc{\bF}{{\bf F}}
\nc{\bS}{{\bf S}}
\nc{\dV}{{\delta V}}
\nc{\dE}{{\delta E}}
\nc{\TT}{{\Theta}}
\nc{\dPsi}{{\delta\Psi}}
\nc{\order}{{\cal O}}
\nc{\Rout}{R_{\rm out}}
\nc{\eplus}{e_+}
\nc{\eminus}{e_-}
\nc{\epm}{e_\pm}
\nc{\eps}{\varepsilon}
\nc{\vnabla}{{\vec\nabla}}
\nc{\G}{\Gamma}
\nc{\w}{\omega}
\nc{\mh}{h}
\nc{\mg}{g}
\nc{\vphi}{\varphi}
\nc{\tlambda}{\tilde\lambda}
\nc{\be}{\begin{equation}}
\nc{\ee}{\end{equation}}
\nc{\ba}{\begin{eqnarray}}
\nc{\ea}{\end{eqnarray}}

\nc{\g}{\gamma}
\nc{\ol}{\overline}

\newtheorem{theorem}{Theorem}[section]
\newtheorem{lemma}[theorem]{Lemma}
\newtheorem{prop}[theorem]{Proposition}
\newtheorem{corollary}[theorem]{Corollary}

\nc{\pT}{\partial_T}
\nc{\pz}{\partial_z}
\nc{\pt}{\partial_t}
\nc{\la}{\langle}
\nc{\ra}{\rangle}
\nc{\infint}{\int_{-\infty}^{\infty}}
\nc{\halfwidth}{6.5cm}
\nc{\figwidth}{10cm}
\newcommand{\f}{\frac}

\nc{\nlayers}{L} \nc{\nsectors}{M}
\nc{\indicator}{\mathbf{1}}
\nc{\Rhole}{R_{\rm hole}}
\nc{\Rring}{R_{\rm ring}}
\nc{\neff}{n_{\rm eff}}
\nc{\Frem}{F_{\rm rem}}
\nc{\R}{\mathbb R}
\nc{\C}{\mathbb C}
\nc{\Z}{\mathbb Z}
\nc{\DD}{\Delta}
\nc{\cD}{\mathcal D}
\nc{\lnorm}{\left\|}
\nc{\rnorm}{\right\|}
\nc{\rnormp}{\right\|_{\ell^{p,\eps}}}
\nc{\rar}{\rightarrow}
\nc{\mR}{\mathcal R}
\nc{\oo}{\"o}   

\nc{\os}{\overset{o}}
\begin{document}

\begin{abstract}

	We consider the fourth order Schr\"odinger operator $H=\Delta^2+V(x)$ in three dimensions with real-valued potential $V$.  Let $H_0=\Delta^2$, if $V$ decays sufficiently
	and there are no eigenvalues or resonances in the absolutely continuous spectrum of $H$ then the wave operators $W_{\pm}= s\text{\,--}\lim_{t\to \pm \infty} e^{itH}e^{-itH_0}$ 
	extend to bounded operators on $L^p(\mathbb R^3)$ for all $1<p<\infty$.

\end{abstract}

\title[Wave Operators for fourth order Schr\"odinger]{On the $L^p$ boundedness of the Wave Operators for fourth order Schr\"odinger operators}

\author[Goldberg, Green]{Michael Goldberg and William~R. Green}
\thanks{The  first  author  is  supported  by  Simons  Foundation  Grant  635369. The second author is supported by  Simons  Foundation  Grant  511825. }
\address{Department of Mathematics\\
	University of Cincinnati \\
	Cincinnati, OH 45221 U.S.A.}
\email{goldbeml@ucmail.uc.edu}
 \address{Department of Mathematics\\
Rose-Hulman Institute of Technology \\
Terre Haute, IN 47803, U.S.A.}
\email{green@rose-hulman.edu}


\maketitle

\section{Introduction}

Let $H_0=\Delta^2$ be the free fourth order Schr\"odinger operator on $L^2(\R^3)$.  Let $H:=\Delta^2+V$ for a real-valued, decaying potential $V(x)$.
The wave operators are defined by
$$
	W_{\pm}=s\text{\ --}\lim_{t\to \pm \infty} e^{itH}e^{-itH_0}.
$$
We investigate the $L^p(\R^3)$ boundedness of the wave operators. One use of the $L^p$ boundedness of the wave operators is due to the intertwining identity:
\begin{align}\label{eq:intertwining}
	f(H)P_{ac}(H)=W_\pm f(\Delta^2)W_{\pm}^*.
\end{align}
Here $f$ may be any Borel function and $P_{ac}(H)$ is projection onto the absolutely continuous spectral subspace of $H$.  Using \eqref{eq:intertwining} one may obtain $L^p$-based mapping properties for the more complicated, perturbed operator $f(H)P_{ac}(H)$ from the simpler free operator $f(\Delta^2)$.  The intertwining holds for the class of potentials $V$ which we consider.  The boundedness of the wave operators on $L^p(\R^3)$ for any choice of $p\geq 2$ with the function $f(\cdot)=e^{-it(\cdot)}$ yield the dispersive estimate
\begin{align}
	\|e^{-itH}P_{ac}(H)\|_{L^{p'}\to L^p}\les |t|^{-\frac{3}{4}+\frac{3}{2p} },
\end{align}
where $p'$ is the H\"older conjugate of $p$.  The intertwining identity is a consequence of the asymptotic completeness of the wave operators, \cite{RS3}.  The intertwining identity can also be established for very general operators, \cite{Hor}, without establishing asymptotic completeness.  

The recent work on dispersive estimates for the fourth order operators, \cite{GT4,egt,GG4} which was inspired by the weighted $L^2$ results in \cite{fsy}, suggest that the wave operators should be bounded for a non-trivial range of $p$, which should depend on the regularity of the threshold.  We say that zero energy is regular if there are no non-trivial solutions to $H\psi =0$ with $\psi \in L^\infty(\R^3)$.  Throughout the paper we denote $\la x\ra :=(1+|x|^2)^{\f12}$.

\begin{theorem}\label{thm:main}
	
	If $|V(x)|\les \la x\ra^{-\beta}$ for some $\beta>9$, and there are no embedded eigenvalues in the spectrum of $H$ and that zero energy is regular, then the wave operators $W_\pm$ are bounded on $L^p(\R^3)$ for all $1<p<\infty$.
	
\end{theorem}

The $L^p$ boundedness of the wave operators for the (second-order) Schr\"odinger operators is a well-studied problem with a rich literature,
see \cite{YajWkp1,YajWkp2,YajWkp3,JY2,DF,Miz}.  These results span all dimensions $n\geq 1$ when the threshold energy is regular.  In three spatial dimensions, detailed structure formulas have been obtained for the wave operators, \cite{Bec,BS,BS2}.  The effect of threshold obstructions has been studied \cite{Yaj,FY,YajNew,JY4,GGwaveop,YajNew2,GGwaveop4,EGG,YajNew3}.  Generically, the existence of threshold resonances shrinks the range of $p$, though some range may be recovered if certain orthogonality conditions between the potential and the threshold eigenspace hold.  Similar issues arise with point-interaction potentials, \cite{DMSY}.
However, to the best of the authors' knowledge, there are no investigations of the $L^p$ boundedness of higher order Schr\"odinger operators in the literature to date. The fourth order Schr\"odinger operators are of particular physical interest due to their use in modeling laser beam propagation, see \cite{K,KS}.

Under the assumptions on $V$ in Theorem~\ref{thm:main}, the operator $H$ is self-adjoint and $\sigma_{ac}(\Delta^2)=\sigma_{ac}(H)=[0,\infty)$.  
Lack of embedded eigenvalues is a standard assumption in the analysis of dispersive equations, however it should not be taken for granted. Perturbing $\Delta^2$ may induce positive eigenvalues even for smooth potentials.  We note that \cite{soffernew} gives an easily applied criterion for the potential to exclude embedded eigenvalues.
The existence of threshold resonances or eigenvalues is known to affect the dispersive estimates, see the work of  Erdo\smash{\u{g}}an, Toprak and the authors in various combinations \cite{GT4,egt,GG4}.  In higher dimensions, the effect of threshold resonances on dispersive estimates between weighted $L^2$ spaces was studied in \cite{FWY}.   Due to the dispersive estimates proven in \cite{egt} in the case of a mild ``resonance of the first kind," we expect the wave operators to be bounded for $1<p<\infty$ in this case.

The paper is organized as follows.  In Section~\ref{sec:prelim} we introduce notation and other results we use repeatedly in our analysis.  In Section~\ref{sec:low} we prove Theorem~\ref{thm:main} when the spectral parameter is in a small neighborhood of zero energy.  In Section~\ref{sec:high} we establish Theorem~\ref{thm:main} holds for the remaining portion of the continuous spectrum that doesn't contain zero and complete the proof of Theorem~\ref{thm:main}.  Finally in Section~\ref{sec:Minv} we prove a technical expansion of the resolvent in a neighborhood of zero energy that is needed for our analysis in Section~\ref{sec:low}.

\section{Preliminaries} \label{sec:prelim}

For the convenience of the reader we begin by defining notation we use throughout the paper, and collect some useful integral estimates.

We use the notation $\beta-$ to denote $\beta-\epsilon$ for an arbitrarily small, but fixed $\epsilon>0$.  Similarly, $\beta+=\beta+\epsilon$.  We also use the polynomially weighted spaces
$$
	L^{2,\sigma}=\{ f\, :\,  \la \cdot \ra^{\sigma}f \in L^2 (\R^3) \}.
$$
We say that an operator $T:L^2\to L^2$ with integral kernel $T(x,y)$ is absolutely bounded if the operator with integral kernel $|T(x,y)|$ is a bounded operator on $L^2$.

Finally, for functions of the spectral variable $\lambda$ we write $h(\lambda)=O_2(\lambda^k)$ to indicate that $|h(\lambda)|\les \lambda^k$ and $|\partial_\lambda^j h(\lambda)|\les \lambda^{k-j}$ for each $0\leq j\leq 2$.  We may also include dependence on the spatial variables in some combination, writing $h(\lambda)=O_2(\lambda^k g(x,z,w,y))$ to denote that $|\partial_\lambda^j h(\lambda)|\les \lambda^{k-j}|g(x,z,w,y)|$ for $0\leq j\leq 2$.  For our purposes, we need operators and functions whose first two derivatives in $\lambda$ may be bounded comparable to division by $\lambda$.  We use the notation $h(\lambda)=O_3(\lambda^k)$ if we have control of the first three derivatives.

Our starting point is the stationary representation for the wave operators,
\begin{align}\label{eq:stat rep}
W_+u=u-\frac{2}{\pi i} \int_0^\infty \lambda^3 R_V^+(\lambda^4)V[R^+-R^-](\lambda^4) u\, d\lambda.
\end{align}
Where $R(\lambda^4)=(\Delta^2-\lambda^4)^{-1}$ and $R_V(\lambda^4)=(\Delta^2+V-\lambda^4)^{-1}$ are the resolvent operators.  The stationary representation in \eqref{eq:stat rep} utilizes the limiting resolvent operators, for $z>0$:
\begin{align*}
	&R^\pm(z)=R(z\pm i0)=\lim_{\epsilon \searrow 0} (\Delta^2-(z\pm i \epsilon))^{-1},\\
	&R_V^\pm(z)=R(z\pm i0)=\lim_{\epsilon \searrow 0} (\Delta^2+V-(z\pm i \epsilon))^{-1}.
\end{align*}
Our proof of Theorem~\ref{thm:main} relies on a careful examination of the integral kernel of the wave operator $W_+(x,y)$.  This suffices to prove the same bounds for $W_-$, as the action of complex conjugation $\mathcal C$ relates the operators through the identity $W_-=\mathcal C^{-1}W_+\mathcal C$.  By duality, when $W_\pm$ are bounded on $L^p$ the adjoint operators $(W_\pm)^\ast$ are bounded on $L^{p'}$ where $p'$ is the H\"older conjugate of $p$.
We use the following lemma on $L^p$-boundedness often. 
\begin{lemma}\cite[Lemma~3.1]{YajNew3}\label{lem:kernels}	
	Suppose that $K$ is an integral operator whose kernel obeys the pointwise bounds
	\begin{align} \label{eqn:kernels}
	|K(x,y)|\les \frac{1}{\la x\ra \la y \ra \la |x|-|y|\ra^{2+\epsilon}}.
	\end{align}
	Then $K$ is a bounded operator on $L^p(\mathbb R^3)$ for $1\leq p\leq \infty$ if $\epsilon>0$, and on $1<p<\infty$ if $\epsilon=0$.  
	
\end{lemma}

We say that a kernel of an operator $K$ is admissible if 
$$
	\sup_{x\in \R^3}\int_{ \R^3} |K(x,y)|\, dy+\sup_{y\in \R^3}\int_{ \R^3} |K(x,y)|\, dx<\infty.
$$
By the Schur Test, an operator with admissible kernel  is bounded on $L^p$ for all $1\leq p\leq \infty$.
There is room in the $\eps = 0$ statement to allow for logarithmic factors, which we prove as a corollary. The logarithmic version is useful to avoid unnecessary integration by parts that would demand faster decay from the potential.

\begin{proof}
The basic principle here is that $K(x,y)$ is a bounded operator on $L^p(\R^3)$ if $(\int_{\R^3} |K(x,y)|^p)^{1/p} dx$ belongs to $L^{p'}(\R^3)$ with respect to $y$.  Since the kernel bound in \eqref{eqn:kernels} is symmetric in $x$ and $y$, the operator is also bounded on $L^{p'}$ and exponents in between.

When $\epsilon = 0$, it is convenient to split the integral into three regions according to whether $|x| > 2|y|$, $|x| <  \frac12 |y|$, or $\frac12|y| \leq |x| \leq 2|y|$.  In the region where $|x|\approx |y|$, which is also symmetric in $x$ and $y$, switching to polar coordinates we see that
$$
	\int_{|x| \approx |y|}|K(x,y)|dx \les  \frac{1}{\la y\ra^2}\int_{ |y|/2}^{2|y|} \frac{r^2}{\la r-|y|\ra^2}\, dr \les \frac{|y|^2}{\la y\ra^2}\int_{ |y|/2}^{2|y|} \frac{1}{\la r-|y|\ra^2}\, dr\les 1,
$$ 
uniformly in $y$. This part of the operator is bounded for any $1 \leq p \leq \infty$.

In the remaining regions one obtains (using that $|\,|x|-|y|\,| \approx |y|$ when $|x|<\f12|y|$)
\[
\int_{|x| < \frac12|y|} \frac{1}{\la x\ra^p\la y\ra^p \la|x| - |y|\ra^{2p}}dx \les \int_{|x| < \frac12|y|} \frac{1}{\la x\ra^p \la y\ra^{3p}}dx \les \la y\ra^{\max(3 -4p, -3p)},
\] 
and (using that $|\,|x|-|y|\,| \approx |x|$ when $|x|>2|y|$)
\[
\int_{|x| > 2|y|} \frac{1}{\la x\ra^p\la y\ra^p\la|x| - |y|\ra^{2p}}dx \les \int_{|x| > 2|y|} \frac{1}{\la x\ra^{3p}\la y\ra^p}dx \les \la y\ra^{3-4p} \text{ if } p>1.
\] 
We see that the large $|x|$ region constrains the range of $p$.  Finally,  we note that $\la y\ra^{\max(3/p - 4, -3)} = \la y\ra^{\max(-3/p' - 1, -3)}$ belongs to $L^{p'}(\R^3)$ for any $1< p < \infty$ so these parts of the operator $K(x,y)$ are bounded on $L^p(\R^3)$ as long as $1 < p < \infty$.

When $\epsilon > 0$ one can check in polar coordinates that
\[
\sup_{y\in \mathbb R^3}\int_{\R^3} \frac{1}{\la x\ra \la y\ra \la |x|-|y|\ra^{2+\epsilon}}dx
= 4\pi \int_0^\infty \frac{r^2}{\la r\ra \la y\ra \la r-|y|\ra^{2+\epsilon}}dr
\les 1.
\]
The last inequality follows by  breaking up into regions based on whether $r\leq \f12|y|$, $r\approx |y|$ or $r\geq 2|y|$, with integrability for large $r$ requiring $\epsilon >0$.  The same bound also holds when $x$ and $y$ are interchanged, hence $K$ has an admissible kernel and is bounded for $1\leq p\leq \infty$.

\end{proof}

\begin{corollary}\label{cor:log bd}
	
	Suppose that $K$ is an integral operator whose kernel obeys the pointwise bounds
	\begin{align*}
	|K(x,y)|\les \frac{\la \log (\la |x|-|y|\ra) \ra }{\la x\ra \la y \ra \la |x|-|y|\ra^{2}}.
	\end{align*}
	Then $K$ is a bounded operator on  $1<p<\infty$.  
	
\end{corollary}

\begin{proof}
	
	The proof mirrors the $\epsilon=0$ case of Lemma~\ref{lem:kernels}.  We use the fact that $\la \log (\la |x|-|y|\ra) \ra \les \la |x|-|y|\ra^{\delta}$, for any choice of $\delta>0$.  In the region when $|x|\approx|y|$ we have (for an arbitrarily small $\delta>0$)
	$$
	\int_{|x| \approx |y|}|K(x,y)|dx \les  \frac{1}{\la y\ra^2}\int_{ |y|/2}^{2|y|} \frac{r^2}{\la r-|y|\ra^{2-\delta}}\, dr \les \frac{|y|^2}{\la y\ra^2}\int_{ |y|/2}^{2|y|} \frac{1}{\la r-|y|\ra^{2-\delta}}\, dr\les 1,
	$$ 
	uniformly in $y$.  This requires only choosing $\delta<1$.
	
	When $|x|<\f12|y|$, we have 
	\[
	\int_{|x| < \frac12|y|} \frac{\la \log (\la |x|-|y|\ra) \ra^p}{\la x\ra^p\la y\ra^p \la|x| - |y|\ra^{2p}}dx \les \int_{|x| < \frac12|y|} \frac{1}{\la x\ra^p \la y\ra^{(3-\delta)p}}dx \les \la y\ra^{\max(3 -(4-\delta)p, -(3-\delta)p)},
	\] 
	and when $|x|>2|y|$
	\[
	\int_{|x| > 2|y|} \frac{\la \log (\la |x|-|y|\ra) \ra^p}{\la x\ra^p\la y\ra^p\la|x| - |y|\ra^{2p}}dx \les \int_{|x| > 2|y|} \frac{1}{\la x\ra^{(3-\delta)p}\la y\ra^p}dx \les \la y\ra^{3 -(4-\delta)p} \text{ if } p> \frac{3}{3-\delta}.
	\] 
	The expression $\la y\ra^{\max(-3/p' - (1-\delta), -(3-\delta))}$ belongs to $L^{p'}(\R^3)$ provided $p' > \frac{3}{3-\delta}$.  For a given choice of $0< \delta < 1$, we see that the kernel is an $L^p$-bounded operator for $p \in (\frac{3}{3-\delta}, \frac{3}{\delta})$. Every $1 < p < \infty$ belongs to this range for a sufficiently small $\delta$.
	
\end{proof}

\begin{lemma}\label{lem:L2 decay}
	
	We have the bound
	\[
	\Big\|\frac{\la z\ra^{-\beta}}{|x-z|}\Big\|_{L^2_z} \les \frac{1}{\la x\ra},
	\]
	provided $\beta>\f32$.  On the other hand, 
	\[
	\sup_x \Big\|\frac{\la z\ra^{-\beta}}{|x-z|}\Big\|_{L^2_z} \les 1,
	\]
	provided $\beta>\f12$. 
	
\end{lemma}

\begin{proof}
	
	The proof follows by dividing the integral into pieces based on the size of $|z|$ compared to $|x|$.  Alternatively, we apply Lemma~3.8 from \cite{GV} to see that
	$$
		\Big\|\frac{\la z\ra^{-\beta}}{|x-z|}\Big\|_{L^2_z}^2
		=\int_{ \R^3}\frac{\la z\ra^{-2\beta}}{|x-z|^2}\, dz
		\les \left\{\begin{array}{ll}
		\la x\ra^{1-2\beta} & \beta <\f32\\
		\la x\ra^{-2} & \beta >\f32
		\end{array}  \right.,
	$$
	provided $2+2\beta>3$.
	
\end{proof}

\section{Low Energy}\label{sec:low}

The low energy behavior of the wave operators constrains the range of $p$ on which the wave operators are bounded, as we show in Section~\ref{sec:high} that the high energy portion is bounded on the full range of $p$.  In this section we develop an appropriate expansion for $R_V^+(\lambda^4)$ in a sufficiently small neighborhood of $\lambda=0$ and control the resulting terms and their contribution to \eqref{eq:stat rep}.  The techniques vary from a delicate argument using harmonic analysis techniques inspired by Yajima's work, \cite{YajNew3}, to a less delicate argument directly bounding oscillatory integrals.

\begin{theorem}\label{thm:low e}
	
	Under the hypotheses of Theorem~\ref{thm:main}, the low energy portion of the wave operator
	$$
		W_+^{L}u= -\frac{2}{\pi i} \int_0^\infty \lambda^3 \chi(\lambda) R_V^+(\lambda^4)V[R^+-R^-](\lambda^4) u\, d\lambda
	$$
	extends to a bounded operator on $L^p(\R^3)$ for any $1<p<\infty$.
	
\end{theorem}

We note that the low energy behavior of the integral, when $0<\lambda<2\lambda_0\ll 1$ for some sufficiently small constant $\lambda_0$ determines the range of boundedness.  

For the low energy portion we use the smooth cut-off function $\chi(\lambda)$ which is one if $0<\lambda<\lambda_0$ and zero if $\lambda >2\lambda_0$.  We then employ the symmetric resolvent identity to rewrite \eqref{eq:stat rep} slightly.  Defining $U(x)=1$ when $V(x)\geq 0$ and $U(x)=-1$ when $V(x)<0$, $v(x)=|V(x)|^{\f12}$ and $M(\lambda)=U+vR(\lambda^4)v$.  Then, for $\Im(\lambda)\neq 0$ and $\Re(\lambda) > 0$ we may write
$$
	R_V(\lambda^4)=R(\lambda^4)vM(\lambda)^{-1}v
$$
This identity extends to the positive real axis in the expected way,
$$
	R_V^\pm(\lambda^4)=R^\pm(\lambda^4)vM^\pm(\lambda)^{-1}v
$$
with $M^\pm(\lambda) = U + vR^\pm(\lambda^4)v$.

It is well-known in the case of the usual Schr\"odinger operator $-\Delta+V$ that the behavior of  $R_V^\pm$ as $\lambda \to 0^{+}$ is highly dependent on the existence of zero energy resonances or eigenvalues.  This is also true for the fourth order operator, see \cite{GT4,FWY,egt,GG4}.

We use the following identity for the resolvent (following \cite{fsy,egt,GG4})
\be\label{eq:R0lambda}
	R^\pm (  \lambda^4) (x,y)= \frac{1}{2 \lambda^2}  \Bigg( \frac{e^{\pm i \lambda|x-y|}}{4 \pi |x-y|}  -  \frac{e^{-\lambda|x-y|}}{4 \pi |x-y|}  \Bigg).
\ee

This expression has a simple pole at $\lambda = 0$ with residue $a^\pm = \frac{1\pm i}{8\pi}$.  Consequently the power series expansion for $M^\pm(\lambda)$ begins with $\frac{a^{\pm}\|V\|_1}{\lambda}P$, where $P$ is the rank-one projection onto the span of $v$ with kernel
$P(x,y) = \frac{v(x)v(y)}{\|V\|_1}$.  The complementary projection $Q = \mathbbm 1 - P$ appears prominently in some formulas below.

We now collect some useful results on the (limiting) operators.

\begin{prop}\label{prop:Minv}
	If zero is a regular point of the spectrum and if $|V(x)|\les \la x\ra^{-9-}$, then in a sufficiently small neighborhood of zero we have the expansion
	\begin{align}\label{eq:Minv}
	[M^+(\lambda)]^{-1}=QD_0Q+ \lambda M_1 +M_2(\lambda)
	\end{align}
	where $QD_0Q$, and $M_1$ are $\lambda$ independent, absolutely bounded operators, and $M_2(\lambda)$ is a $\lambda$ dependent, absolutely bounded operator satisfying
	$$
		\sum_{j=0}^2 \lambda^{ j} \|\,|\partial_\lambda^j M_2(\lambda)|\,\|_{L^2\to L^2} \les \lambda^2, \quad \lambda>0.
	$$
	
\end{prop}
The proof is technical and we defer it to Section~\ref{sec:Minv} for the convenience of the reader.
The operator $D_0$ is intimately tied to the existence or non-existence of zero energy resonances, \cite{egt}.  For our purposes, we need only that $QD_0Q$ is absolutely bounded and that $Q$ is a projection orthogonal to the span of $v$.  Namely, we use that
$$
	Qv= \int_{ \R^3} Qv(z)\, dz=0.
$$
Substituting the resolvent identity and expansion for $[M^+(\lambda)]^{-1}$ from \eqref{eq:Minv} into \eqref{eq:stat rep} we need to control:
\begin{align}\label{eq:wave Q}
	&\int_0^\infty \lambda^3 R^+(\lambda^4) vQD_0Qv[R^+-R^-](\lambda^4)(x,y)\, d\lambda,\\
	&\int_0^\infty \lambda^4 R^+(\lambda^4) vM_1v[R^+-R^-](\lambda^4)(x,y)\, d\lambda,  \label{eq:wave M1}  \\
	&\int_0^\infty \lambda^3 R^+(\lambda^4) vM_2(\lambda) v[R^+-R^-](\lambda^4)(x,y)\, d\lambda.  \label{eq:wave er}
\end{align}
We show that each term is bounded on $L^p(\R^3)$ for all $1<p<\infty$ in Proposition~\ref{prop:QDQ}, and Lemmas~\ref{lem:Hilbert} and \ref{prop:error} respectively.

We lead with the most delicate term.
\begin{lemma} \label{lem:Hilbert}
Let $M_1$ be a $\lambda$ independent, absolutely bounded operator on $L^2(\R^3)$.  Then the operator $A$ defined by
\begin{equation*}
A = \int_0^\infty \lambda^4 R^+(\lambda^4)vM_1v[R^+ - R^-](\lambda^4)(x,y) \chi(\lambda)\, d\lambda
\end{equation*}
is bounded on $L^p(\R^3)$ for all $1 < p < \infty$ provided $v \in L^2(\R^3)$.
\end{lemma}

Our argument to control this term is adapted from Yajima's work studying the effect of zero energy resonances on the boundedness of the wave operators in three dimensions, \cite{YajNew3}.  This argument takes advantage of the cancellation found in singular integral operators such as the Hilbert transform.

\begin{proof}
Using \eqref{eq:R0lambda} we begin by writing out the full integral defining $A$,
\[
A(x,y) = \int_0^\infty \iint_{\R^6} \frac{e^{i\lambda|x-z|}- e^{-\lambda|x-z|}}{8\pi|x-z|} v(z)M_1(z,w)v(w) \frac{e^{i\lambda|w-y|} - e^{-i\lambda|w-y|}}{8\pi|w-y|} \chi(\lambda) \,dw dz d\lambda.
\]
With the cutoff in $\lambda$ in place, Fubini's theorem allows us to compute the $\lambda$ integral first.
\[
A(x,y) = \iint_{\R^6} \frac{v(z)M_1(z,w)v(w)}{64\pi^2} \int_0^\infty \frac{(e^{i\lambda|x-z|}- e^{-\lambda|x-z|}) (e^{i\lambda|w-y|} - e^{-i\lambda|w-y|})}{|x-z| |w-y|} \chi(\lambda) \, d\lambda.
\]
All that needs to be said about the integral over $(z,w)$ is that $\iint_{R^6} |v(z)M_1(z,w)v(w)| dz dw$ is finite.  This is true because $M_1$ is an absolutely bounded operator and $v\in L^2$. The remaining challenge is to show that for each fixed choice of $z$ and $w$, the
operator
\[
A_{z,w}(x,y) = \int_0^\infty \frac{(e^{i\lambda|x-z|}- e^{-\lambda|x-z|}) (e^{i\lambda|w-y|} - e^{-i\lambda|w-y|})}{|x-z| |w-y|} \chi(\lambda) \, d\lambda
\] 
is bounded on $L^p(\R^3)$ uniformly with respect to $z$ and $w$.  Minkowski's inequality for operator norms then completes the proof.

It is apparent that $A_{z,w}(x,y) = A_{0,0}(x-z, y-w)$.  Thus it sufficies to show that $A_{0,0}$ is bounded on $L^p(\R^3)$ since all other $A_{z,w}$ are derived from it by composition with translations. At all points we get the crude bound 
\[
|A_{0,0}(x,y)| \les \int_0^\infty \min(\lambda, |x|^{-1})\min(\lambda,|y|^{-1}) \chi(\lambda)\,d\lambda \les \la x\ra^{-1} \la y\ra^{-1}.
\]
We note that the kernel $\frac{1}{\la x\ra\la y\ra}$ is admissible when restricted to the set $\{-1 \leq |x| - |y|  \leq 1\} \subset \R^6$. To handle the regions where $|x| - |y|$ is larger than 1, we rewrite the integral once more,
\[
A_{0,0}(x,y) = \frac{1}{|x| |y|}\int_0^\infty \Big(e^{(i|x|+i|y|)\lambda} -e^{(i|x| - i|y|)\lambda} - e^{(-|x|+i|y|)\lambda} + e^{(-|x|-i|y|)\lambda}\Big) \chi(\lambda) \,d\lambda.
\]

An integration by parts yields
\begin{multline*}
A_{0,0}(x,y) = \frac{1}{|x| |y|}\Big( \frac{1}{i|x| + i|y|} - \frac{1}{i|x| - i|y|} - \frac{1}{-|x| + i|y|} + \frac{1}{-|x| - i|y|}\Big)
\\
- \frac{1}{|x| |y|} \int_0^\infty \Big(\frac{e^{(i|x|+i|y|)\lambda}}{i|x| + i|y|} -\frac{e^{(i|x| - i|y|)\lambda}}{i|x|-i|y|} - \frac{e^{(-|x|+i|y|)\lambda}}{-|x|+i|y|} + \frac{e^{(-|x|-i|y|)\lambda}}{-|x|-i|y|}\Big) \chi'(\lambda) \,d\lambda.
\end{multline*}
The boundary term will restrict the range of boundedness as the
integral term generates an admissible kernel. Each one of the four pieces can be integrated by parts repeatedly.  After noting that $\pm|x| \pm i|y|$ has magnitude comparable to $|x| + |y|$, the end result is that this part is
bounded by $\frac{1}{|x||y|} \big(\frac{1}{(|x| + |y|)^3} + \frac{1}{(||x|-|y| |)^3}\big)$.

Suppose $|x| > 1$. Then
\[
\sup_{|x|>1}
\int_{||x| - |y|| > 1} \frac{1}{|x||y|}\bigg(\frac{1}{(|x| + |y|)^3} + \frac{1}{\big(\big||x|-|y|\big|\big)^3}  \bigg)\, dy
\les 1,
\]
with the majority of the integral being supported in the annular region where $|y| \sim |x|$ but $\big||y| - |x|\big| > 1$.
 
If $|x| < 1$, then $|y| > 2|x|$ in order to have $|y| - |x| > 1$. After integrating by parts twice more, using the the Mean Value Theorem
on the function 
$$
	f(s)=\frac{e^{(|x|(\pm s+i(1-s))+i|y|)}}{(|x|(\pm s+i(1-s))+i|y|)^{ 3}}, \quad \text{with} \quad 0\leq s\leq 1,
$$ 
provides for cancellation of the form
\[
\Big| \frac{e^{(i|x|+i|y|)\lambda}}{(i|x| + i|y|)^3} - \frac{e^{(-|x|+i|y|)\lambda}}{(-|x|+i|y|)^3}\Big|,
\Big|\frac{e^{(i|x| - i|y|)\lambda}}{(i|x|-i|y|)^3} - \frac{e^{(-|x|-i|y|)\lambda}}{(-|x|-i|y|)^3}\Big| 
\les |x|\Big(\frac{\lambda}{|y|^3} + \frac{1}{|y|^4}\Big).
\]
So when $|x| < 1$ and $|y|-|x| > 1$, we can find the upper bound
\[
\sup_{x\in \R^3}
\int_{|y| > 1} \int_0^\infty \frac{1}{|x||y|}|x|\Big(\frac{\lambda}{|y|^3} + \frac{1}{|y|^4}\Big) |\chi'''(\lambda)|\,d\lambda dy
\les \int_{|y| > 1} \frac{1}{|y|^4}\,dy \les 1.
\]
Thus the integral remainder part of $A_{0,0}(x,y)$ is integrable in $y$ with a bound independent of the choice of $x \in \R^3$. An identical argument with the variables exchanged completes the proof that this is an admissible kernel.

Finally we must determine the $L^p$ boundedness of the operator $\tilde{A}$ with kernel
\[
\tilde{A}(x,y) = \frac{1}{|x| |y|}\Big( \frac{1}{i|x| + i|y|} - \frac{1}{i|x| - i|y|} - \frac{1}{-|x| + i|y|} + \frac{1}{-|x| - i|y|}\Big) = \frac{4i|x|}{|x|^4 - |y|^4}
\]
supported in the region where $|x| - |y|$ is larger than 1.

Given a function $u$ on $\R^3$, we introduce the spherical averages
$M_u(r) = (4\pi)^{-1} \int_{S^2} u(r\omega) d\omega$.  If we further choose $r = |y|^4$ and $s = |x|^4$, then
\begin{equation} \label{eqn:Hilbert}
\|\tilde{A}u\|_p^p = C\int_0^\infty s^{-\frac14}s^{\frac{p}{4}}\Big(\int_{|\sqrt[4]{r} - \sqrt[4]{s}| > 1} \frac{r^{-\frac14}}{s-r}M_u(\sqrt[4]{r})\,dr\Big)^p\,ds.
\end{equation}
The inner integral is more or less a truncation of the Hilbert Transform of $r^{-\f14}M_u(\sqrt[4]{r})$.  In fact, if $0 \leq s < 1$, then it is exactly a centered truncation of the Hilbert Transform as long as we define $M_u(r) = 0$ for all $r<0$.   

When $s \geq 1$, the integral is truncated away from the interval $(\sqrt[4]{s}-1)^4 < r < (\sqrt[4]{s}+1)^4$.  The interval $(\sqrt[4]{s}-1)^4 < r < (\sqrt[4]{s}+1)^4$  can be subdivided into two pieces according to whether $r$ is greater than $2s - (\sqrt[4]{s}-1)^4$.  The left-hand piece, when $r<2s - (\sqrt[4]{s}-1)^4$, is exactly centered around $r=s$.  Let us briefly denote the right-hand piece as $r \in I_s=(2s - (\sqrt[4]{s}-1)^4,(\sqrt[4]{s}+1)^4)$.

We observe that $I_s$ is contained inside the interval $[s, 16s]$. Furthermore, it follows from $2s-(\sqrt[4]{s}-1)^4<r<(\sqrt[4]{s}+1)^4 $ and the binomial theorem that $s^{3/4} < r-s < 15s^{3/4}$ for all $r \in I_s$. This means that an operator with the integral kernel $K(s,r) = \frac{1}{s-r}\big|_{r\in I_s}$ will be bounded pointwise by the centered Hardy-Littlewood maximal function. In other words, we can guarantee the bound
\[
\Big|\int_{|\sqrt[4]{r} - \sqrt[4]{s}| > 1} \frac{r^{-\frac14}}{s-r}M_u(\sqrt[4]{r})\,dr \Big| \les H^*[r^{-1/4}M_u(\sqrt[4]{r})](s) + \mathcal{M}[r^{-1/4}M_u(\sqrt[4]{r})](s),
\]
where $H^*$ is the maximal truncated Hilbert Transform, and $\mathcal{M}$ is the centered Hardy-Littlewood maximal function.

For all $1 < p < \infty$, the inequality $-1 < \frac{p-1}{4} < p-1$ is satisfied, hence the weighted measure $s^{\frac{p-1}{4}} ds$ belongs to the $A_p$ class. Both the maximal Hilbert Transform and the Hardy-Littlewood Maximal function are bounded operators on $L^p(s^{\frac{p-1}{4}}ds)$, (c.f. Theorems 5.1.12 and 7.1.9 in \cite{Graf}).  Applying this to~\eqref{eqn:Hilbert} yields
\[
\|\tilde{A}u\|_p^p \les \int_0^\infty r^{\frac{p-1}{4}} \big(r^{-\frac14}M_u(\sqrt[4]{r})\big)^p\,dr
\leq \int_0^\infty r^{-\frac14} M_{u^p}(\sqrt[4]{r})\,dr = \|u\|_p^p.
\]
\end{proof}

For the remaining pieces, we apply an argument that will not rely on delicate harmonic analysis.  Instead, we use the following lemma repeatedly.

\begin{lemma}\label{lem:near adm}
	
	If $\mathcal E(\lambda)=O_2(\lambda)$, then
	$$
		\int_{0}^{\infty} e^{i\lambda (|x|\pm |y|)} \chi(\lambda) \mathcal E(\lambda)\, d\lambda \les \frac{\la \log (\la |x|\pm |y|\ra) \ra}{\la|x|\pm |y|\ra^2}.
	$$
	
\end{lemma}

\begin{proof}
	
	The boundedness of the integral is clear from the support of the cut-off function $\chi$.
	We separate the integral into two regions based on the size of $\lambda$ and $\frac{1}{|x|\pm |y|}$.  First, we note that
	\begin{align*}
		\bigg|\int_0^{(|x|\pm |y|)^{-1}} e^{i\lambda (|x|\pm |y|)} \chi(\lambda) \mathcal E(\lambda)\, d\lambda \bigg| \les 
		\int_0^{(|x|\pm |y|)^{-1}} \lambda \, d\lambda \les \frac{1}{(|x|\pm |y|)^2}.
	\end{align*}
	On the second region, we integrate by parts two times.  Using the support of $\chi(\lambda)$ we see
	\begin{multline*}
		\bigg|\int_{(|x|\pm |y|)^{-1}}^{\infty } e^{i\lambda (|x|\pm |y|)} \chi(\lambda) \mathcal E(\lambda)\, d\lambda \bigg|\\ 
		\les 	\frac{ | \mathcal E ((|x|\pm |y|)^{-1}
		)| }{(|x|\pm |y|) }+	  \frac{ | \mathcal E^{\prime}((|x|\pm |y|)^{-1}
			)| }{(|x|\pm |y|)^{2}}+\frac{1}{(|x|\pm |y|)^2}
		\int_{(|x|\pm |y|)^{-1}}^{\infty} |\partial_\lambda^2(\chi(\lambda)\mathcal E(\lambda)) |\, d\lambda\\
		 \les \frac{1}{(|x|\pm |y|)^2} + \frac{1}{(|x|\pm |y|)^2}
		 \int_{(|x|\pm |y|)^{-1}}^{\f12}  \lambda^{-1}\, d\lambda \les \frac{1+\log (|x|\pm |y|)}{(|x|\pm |y|)^2}.
	\end{multline*}
	
\end{proof}

\begin{prop}\label{prop:QDQ}
	
	Under the hypotheses of Theorem~\ref{thm:main}, the operators
	$$
		\int_0^\infty \lambda^3 R^+(\lambda^4)vQD_0Qv R^\pm (\lambda^4)\, d\lambda
	$$
	extend to bounded operators on $L^p(\R^3)$ for all $1<p<\infty$.
	
\end{prop}

We control this operator with a series of lemmas.  We begin by writing
\begin{align}\label{eqn:lambda R}
\lambda R^+(\lambda^4, |x-z|)=\frac{e^{i\lambda|x-z|} - e^{-\lambda|x-z|}}{8\pi \lambda|x-z|} = e^{i\lambda|x|}\Big(e^{i\lambda(|x-z| - |x|)}
\frac{1 - e^{(-1+i)\lambda|x-z|}}{8\pi \lambda |x-z|}\Big).
\end{align}
We now take advantage of the orthogonality relationships between $Q$ and $v$.  Namely,  
$$
	\int_{\R^3} Qv(z) f(x)\, dz=0
$$
for any function $f(x)$ of $x$ only.
Accordingly, we may replace $\lambda R^+(\lambda^4, |x-z|)$ with
\begin{align*}
	e^{i\lambda|x|}&\Big(e^{i\lambda(|x-z| - |x|)} \frac{1 - e^{(-1-i)\lambda|x-z|}}{8\pi \lambda |x-z|}- \frac{1-e^{(-1-i)\lambda|x|}}{8\pi \lambda |x|}\Big) \\
	&= e^{i\lambda|x|}\bigg((e^{i\lambda(|x-z| - |x|)} - 1)\frac{1 - e^{(-1-i)\lambda|x-z|}}{8\pi \lambda |x-z|}  
	+ \frac{1 - e^{(-1-i)\lambda|x-z|}}{8\pi \lambda |x-z|}-\frac{1 - e^{(-1-i)\lambda|x|}}{8\pi \lambda |x|}\bigg).
\end{align*}
We now define the auxiliary functions
\begin{align}\label{eq:G def}
	G^\pm(\lambda,x,z)&=(e^{i\lambda(|x-z| - |x|)} - 1)\frac{1 - e^{(-1-i)\lambda|x-z|}}{8\pi \lambda |x-z|},\\
	F^{\pm}(\lambda, x, z)&=\frac{1 - e^{(-1-i)\lambda|x-z|}}{8\pi \lambda |x-z|}-\frac{1 - e^{(-1-i)\lambda|x|}}{8\pi \lambda |x|}.
	\label{eq:F def}
\end{align}

\begin{lemma}\label{lem:G bds}
	
	We have the bounds
	$$
		|\partial_\lambda^k G^\pm(\lambda, x,z)| \les \frac{\la z\ra^{k+1}}{|x-z|\lambda^k}, \qquad k=0,1,2,
	$$
	and
	$$
		|\partial_\lambda^k F^\pm(\lambda, x,z)| \les \frac{\la z\ra }{\min(\la x\ra,|x-z|)  \lambda^k}, \qquad k=0,1,2.
	$$
	
\end{lemma}

\begin{proof}
	
	We begin by  noting the following useful bounds.  Given $A > 0$,
	\begin{equation}\label{eq:derivs}
	\begin{aligned}
	\Big|  \frac{1- e^{(-1\mp i)\lambda A}}{A}\Big| &\les \frac{\lambda}{\la \lambda A\ra}, \\
	\Big| \Big(\frac{d}{d\lambda}\Big)^j\Big(\frac{1- e^{(-1\mp i)\lambda A}}{A}\Big)\Big| &\les \frac{A^{j-1}}{\la \lambda A\ra^N}
	\text{ for } j \geq 1. 
	\end{aligned}
	\end{equation}	
	Here the second bound takes advantage of the decay of $\exp((-1\mp i)\lambda A)$.
	We now write
	\[
		G^\pm(\lambda,x,z)=\int_0^{|x-z|-|x|} \bigg(i e^{\pm i\lambda s} \frac{1 - e^{(-1\mp i)\lambda|x-z|}}{8\pi |x-z|}\bigg) ds.
	\]
	The desired bounds follow from the product rule using \eqref{eq:derivs}
	and noting that
	$$
		\bigg| \partial_\lambda^k \bigg( \int_0^{|x-z|-|x|} i e^{\pm i\lambda s}\, ds \bigg)\,  \bigg| \les \lambda \la z\ra^{k+1}.
	$$
	To control $F^\pm(\lambda, x, z)$ we write
	$$
		F^{\pm}(\lambda, x, z) =\frac{1 - e^{(-1-i)\lambda|x-z|}}{8\pi \lambda |x-z|}-\frac{1 - e^{(-1-i)\lambda|x|}}{8\pi \lambda |x|}=\frac{1}{\lambda} \big( h(|x-z|)-h(z) \big),
	$$
	where
	$$
		h(s)=\frac{1-e^{(-1\mp i)\lambda s }}{8\pi s}.
	$$
	The Fundamental Theorem of Calculus allows us to write
	\begin{align*}
		F^{\pm}(\lambda, x, z)=\frac{1}{\lambda}\int_{|x|}^{|x-z|} h'(s)\, ds.
	\end{align*}
	Note that
	$$
		h'(s)=\frac{-1+(1-(-1\mp i)s\lambda) e^{(-1\mp i)\lambda s} }{8\pi s^2}.
	$$
	By Taylor expansion, we can see that the numerator is $c\lambda^2 s^2+ O((\lambda s)^3)$ and may be differentiated two times in $\lambda$ with effect comparable to division by $\lambda$.  In all cases, we also have the crude bound of $|\partial_\lambda^k h'(s)|\les s^{-2}$.  Putting this together, we see that
	$$
		|\partial_\lambda^k F^\pm(\lambda, x,z)| \les \int_{|x|}^{|x-z|} \frac{1 }{   \lambda^k}\min \bigg(\lambda, \frac{1}{\lambda s^2}\bigg)\, ds \les \lambda^{-k} \int_{|x|}^{|x-z|} \frac{1}{  s }\, ds.
	$$
	Noting that $\min(|x|,|x-z|)\leq s$ and that the length of the $s$ interval is bounded by $\la z\ra $ completes the proof for $|x|>1$.  If $|x|<1$, we obtain the bound $\la z\ra\lambda^{1-k}$, which allows us to replace the $|x|$ in the denominator with $\la x\ra$.
	
\end{proof}

Propostion~\ref{prop:QDQ} follows from the next discussion above by applying the following  Lemma.

\begin{lemma}\label{double trouble}
	
	Under the hypotheses of Theorem~\ref{thm:main}, we have
	\begin{multline*}
	\bigg|\int_0^\infty e^{i\lambda(|x|\pm |y|)} \lambda \chi(\lambda) (G^+(\lambda,x,z)+F^+(\lambda, x,z))vQD_0Qv (G^\pm(\lambda,w,y)+F^\pm(\lambda, y,z))\, d\lambda\bigg|\\ 
	\les \frac{\la \log (\la |x|\pm |y|\ra) \ra}{ \la x\ra  \la y \ra   \la |x|\pm|y|\ra^{2}},
	\end{multline*}
	which extends to a bounded operator on $L^p$ for $1<p<\infty$.
	
\end{lemma}

\begin{proof}
	
	We may rewrite the $\lambda$ integral as
	$$
		\int_{0}^{\infty} e^{i\lambda (|x|\pm |y|)} \chi(\lambda) \mathcal E(\lambda)\, d\lambda,
	$$
	where (ignoring the operator $vQD_0Qv$ for the moment)
	$$
		\mathcal E(\lambda)=\lambda (G^+(\lambda,x,z)+F^+(\lambda, x,z)) (G^\pm(\lambda,w,y)+F^\pm(\lambda, y,z)).
	$$
	By Lemma~\ref{lem:G bds}, we have
	$$
		\mathcal E(\lambda)=  O_2\bigg(\frac{\lambda \la z\ra^3 \la w\ra^3}{ \min(\la x\ra,|x-z|) \min(\la y\ra,|w-y|)} \bigg).
	$$
	This suffices to apply Lemma~\ref{lem:near adm} to see that
	$$
		\bigg| \int_{0}^{\infty} e^{i\lambda (|x|\pm |y|)} \chi(\lambda) \mathcal E(\lambda)\, d\lambda \bigg|\les \frac{\la z\ra^3 \la w\ra^3}{(|x|\pm |y|)^2\min(\la x\ra,|x-z|) \min(\la y\ra,|w-y|) }.
	$$
	On the other hand, a direct integration yields the bound
	$$
		\bigg| \int_{0}^{\infty} e^{i\lambda (|x|\pm |y|)} \chi(\lambda) \mathcal E(\lambda)\, d\lambda \bigg|\les \frac{\la z\ra^3 \la w\ra^3}{\min(\la x\ra,|x-z|) \min(\la y\ra,|w-y|) }.
	$$
	To account for the effect of the operator $vQD_0Qv$, we note that $QD_0Q$ is absolutely bounded, see Lemma~4.3 of \cite{egt}.  Using Lemma~\ref{lem:L2 decay} we have
	\begin{multline}\label{eq:QDQ int}
		\int_{ \R^6} \frac{ |\la z\ra^3 v(z)QD_0Q(z,w)v(w)\la w\ra^3|}{\la |x|\pm |y|\ra^2\min(\la x\ra,|x-z|) \min(\la y\ra,|y-z|) }\, dz\,dw\\
		\les \frac{\la \log (\la |x|\pm |y|\ra) \ra}{\la |x|\pm |y|\ra^2 } \Big\|\frac{\la z\ra^3 v(z)}{\min(\la x\ra,|x-z|)}\Big\|_{L^2_z} \| |QD_0Q|\|_{L^2\to L^2} \Big\|\frac{\la w\ra^3 v(w)}{\min(\la y\ra,|w-y|)}\Big\|_{L^2_w}\\ \les \frac{\la \log (\la |x|\pm |y|\ra) \ra}{\la |x|\pm |y|\ra^2\la x\ra\la y\ra}.
	\end{multline}
	The claim on $L^p$ boundedness follows by Corollary~\ref{cor:log bd}.

\end{proof}

We may now adapt this argument to the error term  from Proposition~\ref{prop:Minv} to control \eqref{eq:wave er}.

\begin{lemma}\label{prop:error}
	
	Under the hypotheses of Theorem~\ref{thm:main}, we have the bound
	$$
	\bigg| \int_0^\infty \lambda^3 R^+(\lambda^4) vM_2(\lambda)v[R^+-R^-](\lambda^4)(x,y)\, d\lambda\bigg| \les \frac{\la \log (\la |x|\pm |y|\ra) \ra}{ \la x\ra  \la y \ra   \la |x|\pm|y|\ra^{2}},
	$$
	which extends to a bounded operator on $L^p$ for $1<p<\infty$.
	
\end{lemma}

\begin{proof}

	We don't rely on any cancellation between the `+' and `-' resolvents and instead write the integral as
	\begin{multline*}
		\int_0^\infty e^{i\lambda(|x|\pm |y|)}  \lambda^{-1} \chi(\lambda)\\
		e^{i\lambda ( |x-z|-|x| \pm (|y|- |y-w|) )} \bigg(\frac{1 - e^{(-1-i)\lambda|x-z|}}{8\pi \lambda |x-z|}\bigg) vM_2(\lambda)   v \bigg( \frac{1 - e^{(-1\mp i)\lambda|w-y|}}{8\pi \lambda |w-y|}\bigg)\, d\lambda.
	\end{multline*}
	Using \eqref{eq:derivs} and the $\lambda$ smallness of $M_2(\lambda)$ and its derivatives, since  $\||\partial_\lambda^k M_2(\lambda)|\|_{2\to 2}\les \lambda^{2-k}$, suffice to apply Lemma~\ref{lem:near adm}.  When derivatives hit the phase $e^{i\lambda ( |x-z|-|x| \pm (|y|- |y-w|) )}$ the effect is comparable to multiplication by $\la z\ra +\la w\ra$, which is absorbed by the decay of $v(z)$ and $v(w)$ respectively. 
	An analysis as in \eqref{eq:QDQ int} using Lemma~\ref{lem:L2 decay} suffices to establish the pointwise bounds.  Corollary~\ref{cor:log bd} then proves the claim on the range of $L^p$ boundedness.

\end{proof}

We are now ready to prove the low energy portion of the main theorem.

\begin{proof}[Proof of Theorem~\ref{thm:low e}]
	
	Starting from the stationary representation \eqref{eq:stat rep}, using the symmetric resolvent identity in Proposition~\ref{prop:Minv} we need to bound the contribution of three integrals.  Lemma~\ref{lem:Hilbert}   bounds \eqref{eq:wave M1}, Lemma~\ref{double trouble} bounds \eqref{eq:wave Q} bound while Lemma~\ref{prop:error} bounds \eqref{eq:wave er}.
	
\end{proof}

\section{High Energy}\label{sec:high}

In this section we control the high energy portion of the wave operators.  We show that, if the spectral parameter is bounded away from zero, that the wave operators are bounded on the full range of $1\leq p\leq\infty$.  To state our results, we define the complementary cut-off function $\widetilde \chi(\lambda)=1-\chi(\lambda)$, so that $\widetilde \chi(\lambda)$ is supported on $\lambda \geq\lambda_0\gtrsim 1$.

\begin{prop}\label{prop:hi bdd}
	
	Under the hypotheses of Theorem~\ref{thm:main}, the high energy portion of the wave operator
	$$
		W_+^{H}u= -\frac{2}{\pi i} \int_0^\infty \lambda^3 \widetilde\chi(\lambda) R_V^+(\lambda^4)V[R^+-R^-](\lambda^4) u\, d\lambda
	$$
	extends to a bounded operator on $L^p(\R^3)$ for any  $1\leq p\leq \infty$ provided there are no eigenvalues embedded in the spectrum of $H$ and $|V(z)|\les \la z\ra^{-5-}$.
	
\end{prop}

The high energy follows by using the resolvent identity
$$
	R_V^+(\lambda^4)=R^+(\lambda^4)-R^+(\lambda^4)VR_V^+(\lambda)
$$
in the stationary representation, \eqref{eq:stat rep}, of the wave operator to bound
\begin{multline}\label{eqn:hi e}
	\int_0^\infty \lambda^3 \widetilde \chi(\lambda) R_V^+(\lambda^4)V(R^+-R^-)(\lambda^4)\, d\lambda
	=\int_0^\infty \lambda^3 \widetilde \chi(\lambda) R^+(\lambda^4)V(R^+-R^-)(\lambda^4)\, d\lambda
	\\-\int_0^\infty \lambda^3 \widetilde \chi(\lambda)R^+(\lambda^4)V R_V^+(\lambda^4)V(R^+-R^-)(\lambda^4)\, d\lambda.
\end{multline}
For the first integral, using \eqref{eq:R0lambda}, we need to show that
\begin{align}\label{eqn:hi trouble}
	A_L^\pm:=\int_{0}^\infty \frac{\widetilde \chi(\lambda)\chi(\lambda/L) }{\lambda |x-z||z-w|}\bigg( e^{i\lambda(|x-z|\pm |z-y|)} 
	-e^{-\lambda |x-z|}(e^{\pm i\lambda |z-y|}) \bigg)
   \, d\lambda
\end{align}
converges in an appropriate sense.  Then 
$$
	\int_0^\infty \lambda^3 \widetilde \chi(\lambda) R^+(\lambda^4)V(R^+-R^-)(\lambda^4)\, d\lambda= \lim_{L\to\infty} (A_L^+-A_L^-).
$$

\begin{lemma}\label{lem:hi chiL}
	
	The operators $A_L^\pm$ are bounded uniformly on $L^\infty$ and converge uniformly to operators $A^\pm$ that are bounded on $L^\infty$.
	
\end{lemma}

By duality arguments we only consider $p=\infty$.  The same argument applies to the operators $W_-$ and $W_\pm^*$, which implies that this operator is bounded on the full range of $p$.  Due to the similarity of the purely oscillatory portion of $R^\pm$ to the free resolvent for the one-dimensional Schr\"odinger free resolvents, the high energy argument in \cite{DF} has been adapted to control part of the fourth order operator in Lemma~\ref{lem:hi chiL}.

\begin{proof}
	
	Define $\psi_L(\lambda)=\frac{\widetilde \chi(\lambda) \chi_L(\lambda)}{\lambda}$ and let $S_\pm=|x-z| \pm  |y-z|$  The most concerning term is when we have only oscillations.  We need to understand (and bound in $L^\infty$)
	$$
		B_Lg =\int_0^\infty \int_{ \R^6} \psi_L(\lambda) e^{i\lambda S_\pm} \frac{V(z)}{|x-z||z-y|} g(y) \, dz\, dy\, d\lambda,
	$$
	where $g\in L^\infty$.  By Fubini, we may interchange the order of integration and seek to control
	$$
		B_Lg=\int_{\R^6} \widehat{\psi_L}(S_\pm) \frac{V(z)}{|x-z||z-y|} g(y) \, dz\, dy.
	$$
	We show that $B_Lg$ is a Cauchy sequence of operators on $L^\infty$, to that end we note some properties of $\widehat{\psi_L}$, namely that
	$$ 
		|\widehat{\psi_L}(k)| \les \min(   \la \log |k|\ra, | k|^{-3}).
	$$
	The first estimate, good for $L^{-1} \leq |k| \leq 1$, is obtained by breaking the interal into two pieces. The 
integral over $\lambda < |k|^{-1}$ is bounded by $\la \log |k|\ra$.  Then the integral over $|k|^{-1} < \lambda$ can be integrated by parts once to see that it is bounded by 1.  The second estimate follows from integrating by parts repeatedly.   Similar analysis shows that
	$$
		|\widehat{\psi_L}(k) - \widehat{\psi_M}(k)| \les \min(\la \log(|k|L)\ra,  (|k|L)^{-3}) \quad \text{for all } M > L.
	$$
	This can be seen by a change of variables $s=\lambda L$ and the previous analysis, noting that $1\les s\les \frac{M}{L}$.
	For any fixed choice of $x$ and $z$, we can write $y$ in polar coordinates centered at $z$ and estimate
	$$
		\int_{\R^3} \frac{|\widehat{\psi_L}(|x-z| \pm |z-y|)|}{|z-y|} dy =4\pi \int_0^\infty r\,|\widehat{\psi_L}(|x-z|\pm r)| \les \la |x-z|\ra^{\mp 1}.
	$$
	The most delicate piece is in the `-' case when $r\approx |x-z|$, one considers cases based on the size of $|\,|x-z|-r|$ and $\frac{1}{2}$.	
	That leads to the bound
	$$
		\int_{\R^6} \frac{|\widehat{\psi_L}(|x-z| \pm |z-y|) V(z)|}{|x-z| |z-y|} dy dz \les \la x\ra^{-(1 \pm 1)}.
	$$
	That is bounded uniformly in both $x$ and $L$, which shows that $B_L$ are bounded operators on $L^\infty(\R^3)$ uniformly in $L$.  Repeating the calculation with a difference between $B_L$ and $B_M$ leads to the estimate
	$$
		\int_{\R^3} \frac{|\widehat{\psi_L}(S) - \widehat{\psi_M}(S)|}{|z-y|} dy = \int_0^\infty r\, |\widehat{\psi_L}(|x-z|\pm r) - \widehat{\psi_M}(|x-z| \pm r)| \les L^{-1} \la |x-z|\ra^{\mp 1}.
	$$
	Since the integral over $z$ is the same as before, except with an additional factor of $L^{-1}$, the family of operators $B_L$ converges in norm as $L \to \infty$.

	The remaining term with $\exp(-\lambda |x-z|)e^{\pm i\lambda|z-y|}$ has an argument that is less delicate.  Utilizing the support of $\widetilde\chi(\lambda)$, repeated integration by parts shows that
	$$
		\Big| \int_0^\infty e^{\pm i\lambda|z-y|} \frac{e^{-\lambda|x-z|}\widetilde{\chi}(\lambda) \chi(\lambda/L)}{\lambda} d\lambda \Big| \les \frac{e^{-|x-z|/2}\la |x-z|\ra^3}{|x-z| \la |z-y|\ra^3}. 
	$$
	One is then left to evaluate the integral
	$$
		\int_{\R^6} \frac{e^{-|x-z|/2}\la |x-z|\ra^3 |V(z)|}{|x-z|^2 |z-y| \la |z-y|\ra^3} dy dz 
		\les \int_{\R^3} \frac{e^{-|x-z|/2}\la |x-z|\ra^3}{|x-z|} dz \les \la x\ra^{-\beta} \les 1
	$$	
	uniformly in $L$.  The bounds are independent of $L$, hence we may conclude that as $L\to\infty$ the bound $\|Ag\|\les  \|g\|_\infty$ holds by the dominated convergence theorem.  
\end{proof}

We use the limiting absorption principle for the second term in \eqref{eqn:hi e}.
\begin{theorem} \label{th:LAP}\cite[Theorem~2.23]{fsy} Let $|V(x)|\les \la x \ra ^{-k-1}$. Then for any $\sigma>k+1/2$, $\partial_z^k R_V(z) \in \mathcal{B}(L^{2,\sigma}(\R^d), L^{2,-\sigma}(\R^d))$ is continuous for $z \notin {0}\cup \Sigma$. Further, 
	\begin{align*} 
	\|\partial_z^k  R_V(z) \|_{L^{2,\sigma}(\R^d) \rar L^{2,-\sigma}(\R^d)} \les z^{-(3+3k)/4}. 
	\end{align*}
\end{theorem}	
The major consequence of the limiting absorption principle that we use in our argument are the operator bounds:
$$
	\|\partial_\lambda^k  R_V(\lambda^4) \|_{L^{2,\sigma}  \rar L^{2,-\sigma} } \les \lambda^{-3} \quad \text{for each } k\geq 0, \quad\text{provided } \sigma>k+\f12.
$$

\begin{lemma}\label{lem:hi2}
	
	Under the hypotheses of Theorem~\ref{thm:main}, the operator defined by
	$$
		\int_0^\infty \lambda^3 \widetilde \chi(\lambda)R^+(\lambda^4)V R_V^+(\lambda^4)V(R^+-R^-)(\lambda^4)\, d\lambda
	$$
	has an admissible kernel and hence
	extends to a bounded operator on $L^p(\R^3)$ for $1\leq p\leq \infty$.
	
\end{lemma}

\begin{proof}
	
	We first consider the bound taken by taking the absolute value of the integrand.  For this we use \eqref{eq:R0lambda} to make the crude bound 
	$$
		|R^\pm (\lambda^4)(x,y)|\les \frac{1}{\lambda^2 |x-y|}.
	$$
	Then, we may use the limiting absorption principle to see
	\begin{multline*}
		\bigg|\int_0^\infty \lambda^3 \widetilde \chi(\lambda)R^+(\lambda^4)V R_V^+(\lambda^4)V(R^+-R^-)(\lambda^4)\, d\lambda \bigg|\\
		\les \int_1^\infty \lambda^3 \|R^+(\lambda^4)V\|_{L^{2,\f12+}}
		\|  R_V(\lambda^4) \|_{L^{2,\f12+}  \rar L^{2,-\f12-}}\|VR^\pm(\lambda^4) \|_{L^{2,\f12+}}\, d\lambda\\
		\les \int_1^\infty \lambda^{-4} \bigg\| \frac{|V(z)|\la z\ra^{\f12+}}{|x-z|} \bigg\|_{L^2_z}\bigg\| \frac{|V(w)|\la w\ra^{\f12+}}{|w-y|} \bigg\|_{L^2_w}\, d\lambda \les \frac{1}{\la x\ra \la y \ra}.
	\end{multline*}
	Here we used the first claim in Lemma~\ref{lem:L2 decay} to push forward the decay in $x$ and $y$.

When $|x| \pm |y| \gtrsim 1$, there is much to be gained from integration by parts. After three integrations by parts the kernel will be bounded by
$$
	\frac{1}{ (|x|\pm |y|)^3}\bigg|  \int_0^\infty  
	\partial_\lambda^3 \bigg( \lambda^{-1} \widetilde\chi(\lambda)\mathcal E_{x,y}^\pm (\lambda) \bigg) \, d\lambda
	\bigg|,
$$
where  
\begin{multline}\label{eq:Exy defn2}
	\mathcal E_{x,y}^\pm (\lambda)=e^{i\lambda ( |x-z|-|x| \pm (|y|- |y-w|) )}\\ 
	\times\bigg(\frac{1-e^{(-1+i)\lambda |x-z|}}{|x-z|}\bigg)V(z) R_V(\lambda^4)(z,w) V(w) \bigg(\frac{1-e^{(-1\pm i)\lambda |z-w|}}{|z-w|}\bigg).
\end{multline}
Let $k_j\in \mathbb N\cup \{0\}$ with $k_1+k_2+k_3+k_4=3$, then the integral is controlled by
\begin{multline*}
	\int_1^\infty  \bigg\| \partial_\lambda^{k_1} e^{i\lambda ( |x-z|-|x|) } \bigg(\frac{1-e^{(-1+i)\lambda |x-z|}}{|x-z|}\bigg)V(z) \la z\ra^{k_2+\f12+}\bigg\|_{L^2_z}\
	\bigg\| \partial_\lambda^{k_2} R_V(\lambda^4)   \bigg\|_{L^{2,k_2+\f12+}_w\to L^{2,-k_2-\f12-}_z}\\
	\bigg\| \partial_\lambda^{k_3} e^{i\lambda (   \pm (|y|- |y-w|) )}V(w) \bigg(\frac{1-e^{(-1\pm i)\lambda |z-w|}}{|z-w|}\bigg) \la w\ra^{k_2+\f12+}	\bigg\|_{L^2_w} \lambda^{-1-k_4}\, d\lambda\\
	\les \int_1^\infty \lambda^{-4} \bigg\| \frac{|V(z)|\la z\ra^{\f72+}}{|x-z|} \bigg\|_{L^2_z}\bigg\| \frac{|V(w)|\la w\ra^{\f72+}}{|w-y|} \bigg\|_{L^2_w}\les \frac{1}{\la x\ra \la y \ra}.
\end{multline*}
Provided $|V(z)|\les \la z\ra^{- 5 -}$, the second bound in Lemma~\ref{lem:L2 decay} may be used at the end.  
Combining this with the previous bound shows that the kernel of the operator is dominated by
$$
\frac{1}{\la |x|\pm |y|\ra^3 \la x\ra \la y \ra},
$$
and hence is admissible by Lemma~\ref{lem:kernels}.

\end{proof}


\begin{proof}[Proof of Proposition~\ref{prop:hi bdd}]
	
	Using the representation \eqref{eqn:hi e}, the claim follows from Lemmas~\ref{lem:hi chiL} and \ref{lem:hi2}.
	
\end{proof}

Now we can assemble all of the pieces of the main Theorem.

\begin{proof}[Proof of Theorem~\ref{thm:main}]

Using the stationary representation, \eqref{eq:stat rep}, and noting that $\chi(\lambda)+\widetilde \chi(\lambda)=1$ on $[0,\infty)$, we have
$$
	W_+u=u+W_+^Lu+W_+^Hu
$$
The identity operator is bounded on $L^p$,
Theorem~\ref{thm:low e} and Proposition~\ref{prop:hi bdd} establish boundedness of the low and energy contributions respectively.

\end{proof}

\section{Proof of Proposition~\ref{prop:Minv}}\label{sec:Minv}

We include the proof of Proposition~\ref{prop:Minv} here for completeness.  Similar expansions, using the Jensen-Nenciu method \cite{JN}, are obtained by Erd\smash{\u{o}}gan, Toprak and the second author in \cite{egt} and the authors in \cite{GG4}.  Neither of these is quite sufficient for our needs.  We begin by collecting useful results.

	\begin{lemma}\label{lem:M_exp} For  $0<\lambda<1$ define  $M^{+}(\lambda) = U + v R^{+}(\lambda^4) v $. Let $P=v\langle \cdot, v\rangle \|V\|_1^{-1}$ denote the orthogonal projection onto the span of $v$.  We have
		\begin{align}  
		\label{Mexp} M^{+}(\lambda)&= A^{+}(\lambda)  +\lambda \widetilde M_1^+ +\widetilde M_2^+(\lambda),\\ 
		\label{Apm}
		A^{+}(\lambda)&= \frac{\|V\|_1 a^+}{\lambda} P+T,
		\end{align} 
		where $T:=U+vG_0v$,  and
		$$
			\sum_{j=0}^2 \lambda^{ j} \|\,|\partial_\lambda^j \widetilde M_{2 }^+(\lambda)|\,\|_{L^2\to L^2} \les \lambda^2,
		$$
		provided that $v(x)\les \la x \ra^{\f92-}$. 
	\end{lemma} 	

\begin{proof}
	
	This is essentially a subcase of the results in Lemma~4.1 in \cite{egt}.  For completeness, we verify that the third derivative of the error term is bounded as claimed.
	
	Nothing that when $\lambda|x-y|< 1$, using the series expansion of \eqref{eq:R0lambda} one obtains
	$$
		M^+(\lambda) (x,y)= U+vG_0v+ v\frac{a^{+}}{\lambda}v   + a_1^{+}  \lambda vG_1v+O_2(\lambda^2 |x-y|^3),\,\,\,\,\lambda|x-y|<1,
	$$
	where $G_0(x,y)=-\frac{|x-y|}{8\pi}$, $G_1=|x-y|^2$ and $a^+, a_1^+$ are non-zero constants whose exact values are unimportant for our purposes.
	
	When $\lambda|x-y|\geq 1$, we instead note that
	$$
	|\partial_\lambda ^j R^+ (\lambda^4)(x,y)|\les \left\{ \begin{array}{ll} \lambda^{-1} & j=0\\
	\frac{|x-y|^{j-1}}{\lambda^2} & j\geq 1 \end{array}
	\right.=O_2(\lambda^2 |x-y|^3).
	$$
	Since $\lambda |x-y|\geq 1$ we may freely multiply this bound by powers of $\lambda |x-y|$ as needed.  	So that
	\begin{align*}
		\widetilde M_2^+(\lambda)=vR^+(\lambda^4)v-\bigg(  vG_0v+ v\frac{a^{+}}{\lambda}v   + a_1^{+}  \lambda vG_1v \bigg)=O_2(\lambda^2 |x-y|^3).
	\end{align*}
	The decay rate on $v$ is required   to show ensure that $v(x)|x-y|^3v(y)$ is a Hilbert-Schmidt kernel.
	
\end{proof}	
 
An application of the Feshbach formula, see Lemma~4.5 of \cite{egt}, shows that
\begin{align} \label{A inverse}
(A^+(\lambda))^{-1}= QD_0Q + g^{+}(\lambda)  S, 
\end{align}
where $g^{+}(\lambda)=(\frac{ a^+ \|V\|_1}{\lambda} +c)^{-1} $ for some $c\in \R$, and $S$ is an absolutely bounded operator.
Noting that 
$$
	g^+(\lambda)=\frac{\lambda}{a^+\|V\|_1}-\frac{c\lambda^2}{(a^+\|V\|_1)^2}+O_3(\lambda^3),
$$  
we may write
$$
	A^+(\lambda)^{-1}= QD_0Q+\lambda A_1^++O_3(\lambda^2).
$$

The Proposition follows via a Neumann series expansion.
\begin{align*}
	(M^+(\lambda))^{-1}&=(A^{+}(\lambda)    +\lambda\widetilde  M_1^+ +\widetilde M_2^+(\lambda))^{-1}\\
	&=A^+(\lambda)^{-1}(\mathbbm 1  +\lambda \widetilde M_1^+A^+(\lambda)^{-1} +\widetilde M_2^+(\lambda)A^+(\lambda)^{-1})^{-1}\\
	&=QD_0Q+\lambda M_1+M_2(\lambda).
\end{align*}

\end{document}